\documentclass{article}
\title{Instability of Ginzburg-Landau Vortices on Manifolds}
\author{Ko-Shin Chen \\Indiana University, Bloomington \\ koshchen@indiana.edu}
\date{\today}
\usepackage{amsmath,amsthm,amssymb}
\theoremstyle{plain}
\newtheorem{prop}{Proposition}[section]
\newtheorem{lem}{Lemma}[section]
\newtheorem{thm}{Theorem}[section]
\theoremstyle{remark}
\newtheorem{rk}{\bf Remark}[section]
\numberwithin{equation}{section}
\usepackage{graphicx}

\begin{document}
\maketitle

\begin{abstract}
We investigate two settings of Ginzburg-Landau posed on a manifold
where vortices are unstable. The first is an instability result for
critical points with vortices of the Ginzburg-Landau energy posed on
a simply connected, compact, closed $2$-manifold. The second is a
vortex annihilation result for the Ginzburg-Landau heat flow posed
on certain surfaces of revolution with boundary.
\end{abstract}

\section{Introduction}

    In this paper we consider the Ginzburg-Landau energy posed on a $2$-manifold.
We will present two results, one for critical points of the Ginzburg-Landau energy 
and one for the Ginzburg-Landau heat flow, both showing the non-existence of stable 
vortex solutions under certain geometric assumptions on the manifold. We say a critical
point is unstable if there is a direction in which the second variation of the energy is
negative. For the heat flow, we will show that all initial data, even those containing 
vortices, will eventually converge to a vortex-free solution.

    Let $E_\varepsilon$ be Ginzburg-Landau energy on an orientable manifold $\cal M$ equipped with a metric $g$ for $u: \cal M \rightarrow \mathbb C$,
\[
        E_\varepsilon(u) = \frac{1}{2} \int_{\cal M} ||\nabla_g u||_g^2 + \frac{(1-|u|^2)^2}{2\varepsilon^2} dv_g.
\]
There is a vast literature on Ginzburg-Landau, but we review here
just a few of the results most closely related to our investigation.
When $\cal M$ is a bounded domain $\Omega \subset \mathbb R^2$, and
under an $S^1$-valued Dirichlet condition, Bethuel, Brezis and
H\'{e}lein establish in \cite{BBH} that vortices of minimizers
converge as $\varepsilon \rightarrow 0$ to a finite set of points or
limiting vortices $\{a_i\}$. Here vortices refer to zeros of the
order parameter $u_{\varepsilon}$ carrying nonzero degree. Moreover,
the limiting vortex locations  $\{a_i\}$ will minimize a
renormalized energy $W$. Another important result comes in \cite{JM}
where for $u:{\cal M} = \Omega \subset \mathbb R^n \rightarrow
\mathbb R^N$, Jimbo and Morita prove that under homogeneous Neumann
boundary conditions, if $\Omega$ is convex, the only stable critical
points are constants for any $\varepsilon > 0$.

Most important to our work on stability of critical points is the
work of Serfaty in \cite {S} on Ginzburg-Landau in simply connected
planar domains. Here she shows that there is no nonconstant stable
critical point of $E_\varepsilon$ with homogeneous Neumann boundary
conditions for $\varepsilon$ small. To achieve this, she shows that
the renormalized energy has no stable critical points. Then using
her theory of ``$C^2$-Gamma convergence," she argues that there must
be unstable directions for the Hessian of $E_{\varepsilon}$ as well
for $\varepsilon$ small. Our first main result (Theorem \ref{main})
in this paper is that for compact, simply connected $2$-manifolds
without boundary, any critical points must be unstable when
$\varepsilon$ is small if at least one limiting vortex is located at
a point of positive Gauss curvature. Furthermore, if one
additionally assumes that ${\cal M}$ is a surface of revolution with
non-zero Gauss curvature at at least one of the poles, then we argue
that all critical points are unstable for $\varepsilon$ small,
regardless of the curvature of the manifold at the limiting vortex
locations (Theorem \ref{thm3}). To prove this, we will apply
Serfaty's abstract result in \cite{S} (see Theorem \ref{S} below),
showing again that the renormalized energy has no stable critical
points on such manifolds. For Ginzburg-Landau posed on a
$2$-manifold, Baraket generalizes the work of \cite{BBH} to identify
the renormalized energy on compact $2$-manifolds without boundary in
\cite{B}. We should perhaps note that for Ginzburg-Landau in
$3$-dimensional domains, there do exist stable vortex solutions
(\cite{MSZ}). This analysis will be presented in Section 2.

   The second setting we consider is the heat flow for the Ginzburg-Landau energy, 
with $\varepsilon=1$, on surfaces of revolution ${\cal M}$ with boundary:
\begin{equation}
                \left\{\begin{array}{ll}
                 u_t - \bigtriangleup _{\cal M} u = (1-|u|^2)u & \mbox{in $\cal M \times \mathbb R_+$}, \\
                 u = e & \mbox{on $\partial \cal M \times \mathbb R_+$}, \\
                 u = u_0 & \mbox{on $\cal M \times $$\{ t = 0 \}. $}
                \end{array} \right. \notag
\end{equation}
Here $u:\cal M \times \mathbb R_+ \rightarrow \mathbb R$$^2$, and
$e$ is any constant unit vector. We allow the compatible initial
data $u_0$ to have any number of vortices though necessarily the
total degree $\sum d_i =0$ in light of the Dirichlet condition. We
want to find conditions on $\cal M$ such that as $t \rightarrow
\infty$, all vortices are annihilated. When ${\cal M} = \mathbb
R^2$, it has been shown in \cite{BCPS} that if $u_0$ is close to $e$
at infinity in some sense, all vortices of $u$ disappear after a
finite time. As in \cite{BCPS}, we will derive a Pohozaev-type
inequality on surfaces (Lemma \ref{Pohozaev}) to prove a  similar
result when $\cal M$ is a simply connected surface of revolution
satisfying an extra geometric assumption that is unrelated to
curvature, see Theorem \ref{KoShin}. This work is presented in
Section 3.

\section{Instability of Critical Points on a Compact Surface}

In this section we take $\cal M$ to be a simply connected compact surface without boundary, 
and $g$ be a metric on $\cal M$. Consider the Ginzburg-Landay energy on $\cal M$,
\begin{equation}\label{E}
    E_\varepsilon(u) = \frac{1}{2} \int_{\cal M} ||\nabla_g u||_g^2 + \frac{(1-|u|^2)^2}{2\varepsilon^2} dv_g
\end{equation}
where $u \in H^1(\cal M, \mathbb C)$. Let $u_\varepsilon$ be the critical point of \eqref{E}, then $u_\varepsilon$ satisfies
\begin{equation} \label{GLM}
    -\bigtriangleup_g u_\varepsilon = \frac{u_\epsilon (1-|u_\epsilon|^2)}{\varepsilon^2} \mbox{ in } \cal M
\end{equation}

In \cite{BBH} where $\cal M$ is a bounded planer domain, Bethuel,
Brezis and H\'{e}lein prove that under an $S^1$-valued Dirichlet
boundary condition, critical points $u_\varepsilon$ of  \eqref{E}
converges to a limiting map $u_*$ strongly in $C^k_{loc}({\cal M}
\setminus \bigcup_{i=1}^{n} a_i)$ for every integer $k$ and in
$C^{1,\alpha}_{loc}(\bar{\cal M} \setminus \bigcup_{i=1}^{n} a_i)$ for
$\alpha <1$ where $\{a_i\}$ is a finite set.  This result has been
generalized to a compact manifold $\cal M$ without boundary, cf.
\cite{B,CS} and has since been refined, see e.g. \cite{JS} and
\cite{SS}. Thus, we have:

\begin{prop}\label{Baraket}
    Let $\{u_\varepsilon\}$ be a sequence of critical points of $E_{\varepsilon}$ 
with $E_\varepsilon(u_\epsilon) \leq C|\log \varepsilon|$ for some constant $C>0$. 
Then up to extraction of a subsequence, there exists a finite set of points 
$a_1,...,a_n$ in $\cal M$ such that $u_\varepsilon \rightarrow u_*$ strongly 
in $W^{1,p}(\cal M)$ for $p<2$ and in $H^1_{loc}$$(\cal M$$ \setminus \bigcup_{i=1}^{n} a_i)$.
\end{prop}

We will refer to these points $a_1,...,a_n$ as limiting vortices associates with the 
sequence $\{u_\varepsilon\}$.The same result holds on a compact manifold with 
modifications, see Proposition 5.5 in \cite{CS}.

From the Uniformization Theorem, there is a conformal map 
$h: \cal M \rightarrow \mathbb R$$^2 \bigcup\{\infty \}$, so that the metric $g$ is given 
by $e^{2f}(dx_1^2 + dx_2^2)$ for some smooth function $f$. We recall that 
$\bigtriangleup f = -K_{\cal M}e^{2f}$, where $K_{\cal M}$ is the Gauss curvature on 
$\cal M$. Then for $U_{\varepsilon} := u_{\varepsilon} \circ h^{-1}$, \eqref{GLM} transforms to
\begin{equation}\label{GL}
    -\bigtriangleup U_\varepsilon = \frac{e^{2f}}{\varepsilon^2} U_\varepsilon(1-|U_\varepsilon|^2) \mbox{ in } \mathbb R^2.
\end{equation}
We may assume that $h (a_i) \neq \infty$ for all $i$ and denote $b_i := h(a_i)$. 
With a slight abuse of terminology, we will also call the $b_i$'s limiting vortices. 
Then $u_*$ is the harmonic map associated to $(b_i,d_i)$:
\begin{equation}
    u_*(x) = \prod_{i=1}^n (\frac{x-b_i}{|x-b_i|})^{d_i}e^{i\psi} \mbox{ in } \mathbb R^2,
\end{equation}
where $d_i \in \mathbb Z \setminus \{0\}$, $\sum_{i=1}^n d_i =0$,
and $\psi$ is a smooth harmonic function. We also note that the notion 
of convergence linking $\{u_\varepsilon\}$ to 
$\{a_i\}$ is that of convergence of the sequence of Jacobians, namely
\[
curl (iU_\varepsilon, \nabla U_\varepsilon) \rightharpoonup 2\pi 
\sum_{i=1}^n d_i \delta_{b_i}
\] 
in the sense of distributions, where $(\cdot, \cdot)$ denotes the scalar product in $\mathbb C$. 
For Euclidean domains, the proof of this convergence of Jacobians can be found 
in \cite{JS,SS} and the adaptation to the manifold setting is immediate. Moreover, the 
renormalized energy can be defined in the following way:

    Given $\{a_i\}_{i=1}^n \subset \cal M$, let $B_i^g(r)$ be the geodesic ball in $\cal M$ 
centered at $a_i$ with radius $r$, and $B_i(r) = h(B_i^g(r)) \subset \mathbb R^2$. 
Consider $\Phi_r$ which satisfies
\begin{equation}\label{25}
    \left\{\begin{array}{ll}
        \bigtriangleup \Phi_r = 0 & \mbox{in } \mathbb R^2 \setminus \bigcup_{i=1}^n B_i(r)\\
        \Phi_r = const. & \mbox{on each } \partial B_i(r)\\
        \int_{\partial B_i(r)} \frac{\partial \Phi_r}{\partial \nu} = 2\pi d_i & \mbox{for }1 \leq i \leq n.
    \end{array} \right.
\end{equation}
To see the existence of such a solution $\Phi_r$, we first consider a functional
\[
    E(v) = \int_{\cal M} |\nabla_g v|^2 - 2\pi \sum_{i=1}^n d_i v|_{\partial B_i^g(r)}
\]
defined for $v \in H^1(\cal M)$ such that $v|_{\partial B_i^g(r)} = const.$ for each $i$. 
Using the direct method one can show that there exists a minimizer $v_*$ of $E$. 
Then since $g$ is a conformal metric, $v_* \circ h^{-1}$ satisfies \eqref{25}. 
The renormalized energy is defined by
\begin{equation}
    W(\mathbf b) = \lim_{r \rightarrow 0} \frac{1}{2} \int_{\mathbb R^2 \setminus \bigcup_{i=1}^n B_i(r)} |\nabla \Phi_r|^2 dx - \pi \sum_{i=1}^n d_i^2 \log \frac{1}{r},
\end{equation}
where $\mathbf b = (b_1,...,b_n)$.

Finally, using the fact that $\Phi_r(x) \approx \sum_{i=1}^n d_i \log |x-b_i|$ for 
$x \in \mathbb R^2 \setminus \bigcup_{i=1}^n B_i(r)$ when $r \ll 1$, Theorem 2.1 in \cite{B} 
establishes that $W$ can be written as
\begin{equation} \label{W}
    W(\mathbf b) = \pi \sum_{i=1}^{n} d_i^2 f(b_i) - \pi \sum_{i \neq j} d_i d_j \log|b_i-b_j|.
\end{equation}

Our first main result is the following:
\begin{thm}\label{main}
    Let  $\{u_\varepsilon \}$ be a family of solutions to \eqref{GL} such that $E_\varepsilon (u_\varepsilon) \leq C |\log \varepsilon|$, and let $\{a_i\}_{i=1}^n \subset \cal M$ be the limiting vortices for $\{u_\varepsilon\}$. Suppose there exists an  $a_i$ such that the Gauss curvature $K_{\cal M}$ is positive at $a_i$. Then for $\varepsilon$ small enough, $u_\varepsilon$ is unstable.
\end{thm}

\begin{rk}
        Theorem \ref{main} implies that if $\cal M$ has positive curvature 
    everywhere, then there is no stable solution to \eqref{GL} having vortices for
    $\varepsilon$ small enough. Moreover, in this case any solution without any 
    vortices must then be a constant (see Lemma 3.2 in the next section).
    For the special case where $\cal M=S^2$, this instability result
    was first obtained by Contreras, \cite{C}
\end{rk}

\begin{rk}[The Apple Problem]
        If one wants to look for an example of a stable nonconstant critical point, 
    one might consider a surface of revolution $\cal A$ obtained by rotating a 
    smooth curve $\Gamma$ about the $z$-axis shown in Figure 1, so that the 
    shape of $\cal A$ is like an apple. Indeed, one can easily construct a critical 
    point with vortices at $S$ and $N$ (cf. \cite{CS}), but it cannot be stable in 
    view of Theorem \ref{main}, since $K_{\cal A}$ is positive at poles $N$ and $S$. 
    We note that for the 3-D case (solid apple in $\mathbb R^3$), it has been 
    proven in \cite{MSZ} that a critical point with a vortex line through $S$ 
    and $N$ is a local minimizer for $\varepsilon$ small enough.
\end{rk}

\begin{figure}[h]
\graphicspath{{words/}}
\centering
\includegraphics[scale=0.35]{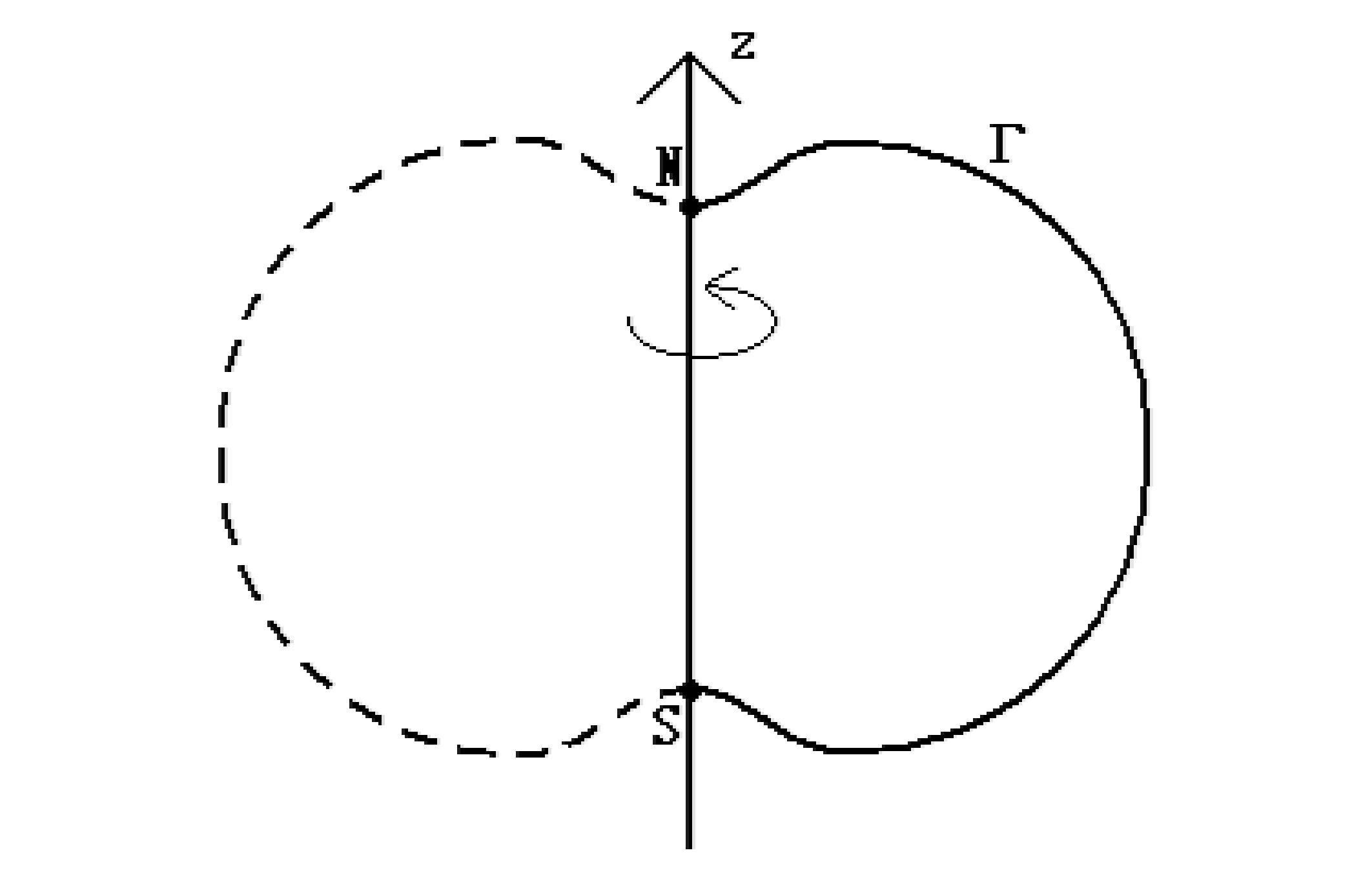}
\caption{Surface of revolution $\cal A$}
\label{eg}
\end{figure}

To prove Theorem \ref{main}, our main tool will be Serfaty's abstract 
result in \cite{S} for any $C^2$ functionals $F_\varepsilon$ (resp. $F$) 
defined over $\cal S$ (resp. $\cal S'$), which is an open set of an affine 
space associated to a Banach space $\cal B$ (resp. $\cal B'$) satisfying a 
kind of \textquotedblleft $C^2$ $\Gamma$-convergence". 
Let $u_\varepsilon \in \cal S$ be a family of critical points of 
$F_\varepsilon$. Assume $u_\varepsilon$ converges to $u \in S'$ in 
some topology. Then denoting by $n_\varepsilon^-$ (resp. $n^-$) the 
dimension of the space spanned by eigenvectors of 
$D^2F_\varepsilon(u_\varepsilon)$ defined over $\cal B$ (resp. $D^2F(u)$ 
defined over $\cal B'$) associated to negative eigenvalues, the theorem states

\begin{thm}[\cite{S}]\label{S}
    Suppose that for any $V \in \cal B'$, there exists $v_\varepsilon(t) \in \cal S$ defined in a neighborhood of $t=0$ such that
    \begin{equation}\label{h1}
        v_\varepsilon(0) = u_\varepsilon,
    \end{equation}
    \begin{equation}\label{h2}
        \partial_t v_\varepsilon(0) \mbox{ is a one-to-one linear map on } \cal B',
    \end{equation}
    \begin{equation}\label{h3}
        \lim_{\varepsilon \rightarrow 0} \frac{d}{dt}|_{t=0}F_\varepsilon(v_\varepsilon(t)) = \frac{d}{dt}|_{t=0}F(u+tV),
    \end{equation}
    \begin{equation}\label{h4}
        \lim_{\varepsilon \rightarrow 0} \frac{d^2}{dt^2}|_{t=0}F_\varepsilon(v_\varepsilon(t)) = \frac{d^2}{dt^2}|_{t=0}F(u+tV).
    \end{equation}
    Then for $\varepsilon$ small enough, we have $n_\varepsilon^- \geq n^-$.
\end{thm}

In the same paper, she applies this result to Ginzburg-Landau 
in bounded domains in $\mathbb R^2$ with homogeneous 
Neumann boundary conditions. In a similar manner, to prove 
Theorem \ref{main}, we apply this approach to Ginzburg-Landau 
on surfaces. That is, using the same notation as above we shall prove

\begin{prop}
    Let $u_\varepsilon$ be a family of critical points of $E_\varepsilon$ 
    such that $E_\varepsilon(u_\varepsilon) \leq C|\log \varepsilon|$, 
    and $b_1, b_2, ..., b_n$ be limiting vortices with total degree zero.
    Then hypotheses \eqref{h1} to \eqref{h4} in Theorem \ref{S} hold 
    for $F_\varepsilon = E_\varepsilon$, $F=W$ and 
    $\cal B' = $$\mathbb V$$= \{ (V_1, V_2, ..., V_n):  
    V_i \in {\mathbb R^2} \; \forall 1 \leq i \leq n \}$.
\end{prop}

\begin{proof}
    Let $u_\varepsilon$ be a family of critical points of 
$E_\varepsilon$ such that $E_\varepsilon(u_\varepsilon) \leq C|\log \varepsilon|$. 
Then from results in \cite{B} (see also \cite{CS}), there exists $\rho >0$ 
small enough  such that $B_i^g(\rho)$ are disjoint with 
$|u_\varepsilon| \geq \frac{1}{2}$ in $\cal M \setminus $$\bigcup_{i=1}^n$
$B_i^g(\rho)$ for $\varepsilon$ small enough.

    The construction of $v_\varepsilon (t)$ is based on Propsition III.1 
in \cite{S}. Let $B_i = B_i(\rho) = h(B_i^g(\rho))$. For a given 
$\mathbf V \in \mathbb V$, we can define a $C^1$ family of diffeomorphisms 
of $\mathbb R^2$, $\chi_t(x) = x + t\mathbf X(x)$, in a neighborhood of $t=0$ 
such that $\mathbf X$ has compact support in a set 
$K \subset \subset \mathbb R^2$ and
    \[
        \mathbf X(x) = V_i \mbox{ in each } B_i.
    \]
    Then we define $\Phi_{0,t}$ by
    \begin{equation}
        \bigtriangleup \Phi_{0,t} = 2\pi \sum_{i=1}^{n}d_i\delta_{b_i(t)} \mbox{ in } \mathbb R^2,
    \end{equation}
so that $\Phi_{0,t}(x) = \sum_{i=1}^n d_i \log |x - b_i(t)|$, and let 
$\Phi_{0,0} = \Phi_0$, where $b_i(t) = \chi_t(b_i)$. Then we denote 
by $\theta_t^i$ the polar coordinate centered at $b_i(t)$, and let
    \[
         \psi_t = \sum_{i=1}^{n} d_i \theta_t^i \circ \chi_t - \sum_{i=1}^{n} d_i \theta_0^i.
    \]
    Then we have
    \begin{equation}\label{*}
        \nabla ^\perp \Phi_0 + \nabla \psi_t = \nabla (\sum_{i=1}^n d_i \theta_t^j \circ \chi_t).
    \end{equation}
    Finally we define $v_\varepsilon(\chi_t(x),t) = u_\varepsilon(x)e^{i\psi_t(x)}$. 
With the same argument as in \cite{S}, one can show that \eqref{h1} 
and \eqref{h2} hold for $v_\varepsilon$. Since $\mathbf X$ is 
compactly supported, the domain of integration reduces from 
$\mathbb R^2$ to a compact set. Consequently, the result of 
product-estimate  derived in \cite{S2} used in the original proof 
can be also applied in our case. Therefore we proceed to verify \eqref{h3}.

    By the change of variables $y = \chi_t(x)$, we have
    \begin{align}
        E_\varepsilon(v_\varepsilon) & = \frac{1}{2} \int_{\mathbb R^2} |\nabla v_\varepsilon(y)|^2 +\frac{e^{2f(y)}}{2 \varepsilon^2} (1-|v_\varepsilon|^2)^2 dy \notag \\
                                            & = \frac{1}{2} \int_{\mathbb R^2} |\nabla (u_\varepsilon e^{i \psi_t})(D \chi_t)^{-1}|^2 +\frac{e^{2f(\chi_t)}}{2 \varepsilon^2} (1-|u_\varepsilon|^2)^2 |Jac \; \chi_t| dx.
    \end{align}

    Noting that $\chi_t$ is the identity map in $\mathbb R^2 \setminus K$ 
and a translation along a constant vector $V_i$ in each $B_i$, we deduce 
that in $\mathbb R^2 \setminus K$ and $\bigcup_{i=1}^n B_i$,
    \begin{equation}\label{Jac}
         \frac{d}{dt} (D\chi_t)^{-1} = \frac{d^2}{dt^2}  (D\chi_t)^{-1} = \frac{d}{dt} |Jac \; \chi_t| = \frac{d^2}{dt^2} |Jac \; \chi_t| = 0.
    \end{equation}

    \begin{flushleft}
    \bf{-Derivative of the potential term:}
    \end{flushleft}
    From \eqref{Jac}, we derive
     \begin{align}
        \frac{d}{dt}|_{t=0} & \int_{\mathbb R^2} \frac{e^{2f(\chi_t)}}{4\varepsilon^2} (1-|u_\varepsilon^2|)^2 |Jac \; \chi_t| dx \notag \\
                                                                   = &\int_{K \setminus \bigcup_{i=1}^n B_i}  \frac{e^{2f(\chi_t)}}{2\varepsilon^2} (1-|u_\varepsilon^2|)^2 \; \nabla f \cdot \mathbf X dx \notag \\
                                                                      & + \int_{K \setminus \bigcup_{i=1}^n B_i}  \frac{e^{2f(\chi_t)}}{2\varepsilon^2} (1-|u_\varepsilon^2|)^2 \frac{d}{dt}|_{t=0}|Jac \; \chi_t| dx
     \end{align}
    Since $\nabla f \cdot \mathbf X$ and 
$\frac{d}{dt}|_{t=0}|Jac \; \chi_t|$ are bounded in 
$K \setminus \bigcup_{i=1}^n B_i$, one can apply Lemma 3.2 
in \cite{B} which asserts that 
$\frac{e^{2f}}{\varepsilon}(1-|u_\varepsilon|^2)^2$ converges 
to a measure supported on $\bigcup_{i=1}^n b_i$. Hence, we have
    \begin{equation}
         \frac{d}{dt}|_{t=0} \int_{\mathbb R^2} \frac{e^{2f(\chi_t)}}{4\varepsilon^2} (1-|u_\varepsilon^2|)^2 |Jac \; \chi_t| dx \rightarrow 0 \mbox{ as } \varepsilon \rightarrow 0.
    \end{equation}
    Similarly,
    \begin{equation}
         \frac{d^2}{dt^2}|_{t=0} \int_{\mathbb R^2} \frac{e^{2f(\chi_t)}}{4\varepsilon^2} (1-|u_\varepsilon^2|)^2 |Jac \; \chi_t| dx \rightarrow 0 \mbox{ as } \varepsilon \rightarrow 0.
    \end{equation}

    \begin{flushleft}
    \bf{-Derivative of the gradient term:}
    \end{flushleft}
    Using \eqref{Jac} we have

    \begin{align}
        \frac{d}{dt}|_{t=0} \frac{1}{2} & \int_{\mathbb R^2} |\nabla (u_\varepsilon e^{i \psi_t})(D \chi_t)^{-1}|^2 |Jac \; \chi_t| dx \notag \\
                                   =  &\int_K iu_\varepsilon \frac{d}{dt}|_{t=0} \nabla \psi_t \cdot \nabla u_\varepsilon + \nabla u_\varepsilon \frac{d}{dt}|_{t=0} (D \chi_t)^{-1} \cdot \nabla u_\varepsilon dx \notag \\
                                       & + \frac{1}{2} \int_K |\nabla u_\varepsilon|^2 \frac{d}{dt}|_{t=0} |Jac \; \chi_t| dx.
    \end{align}
Theorem 2 in \cite{B} asserts that $u_\varepsilon$ 
converges to $u_*$ in $H^1_{loc}({\mathbb R^2} \setminus \bigcup_{i=1}^{n} b_i)$. Moreover, $\nabla u_* = \nabla^{\perp}\Phi_0$. Thus we obtain
    \begin{align}
        \frac{d}{dt}|_{t=0} E_\varepsilon (v_\varepsilon) = & \int_K \nabla ^{\perp} \Phi_0 \frac{d}{dt}|_{t=0}(D \chi_t)^{-1} \cdot  \nabla ^{\perp} \Phi_0 + \frac{d}{dt}|_{t=0} \nabla \psi_t \cdot \nabla u_\varepsilon dx \notag \\
                                                                               & + \int_K |\nabla ^{\perp} \Phi_0|^2 \frac{d}{dt}|_{t=0} |Jac \; \chi_t| dx + o_\varepsilon(1).
    \end{align}
    Using \eqref{Jac} again  and the change of variables $x = \chi_t^{-1}(y)$, for any $0<r<\rho$, we have
    \begin{align}
        \frac{d}{dt}|_{t=0} E_\varepsilon (v_\varepsilon) = & \frac{d}{dt}|_{t=0} \frac{1}{2} \int_{\mathbb R^2 \setminus \bigcup_{i=1}^n B_i(r)} |(\nabla ^{\perp} \Phi_0 + \nabla \psi_t)(D \chi_t)^{-1}|^2 |Jac \; \chi_t| dx + o_\varepsilon(1) \notag \\
                                                                           = & \frac{d}{dt}|_{t=0} \frac{1}{2} \int_{\mathbb R^2 \setminus \bigcup_{i=1}^n B_i(t,r)} |\nabla ^{\perp} \Phi_{0,t}|^2dy + o_\varepsilon(1).
    \end{align}
where $B_i(t,r) = \chi_t(B_i(r))$. The last equality comes from \eqref{*}. Next, define $\Phi_{r,t}$ by
    \begin{equation}
        \left\{\begin{array}{ll}
            \bigtriangleup \Phi_{r,t} = 0 & \mbox{in } \mathbb R^2 \setminus \bigcup_{i=1}^n B_i(t,r)\\
            \Phi_{r,t} = const. & \mbox{on each } \partial B_i(t,r)\\
            \int_{\partial B_i(t,r)} \frac{\partial \Phi_{r,t}}{\partial \nu} = 2\pi d_i & \mbox{for }1 \leq i \leq n.
        \end{array} \right.
    \end{equation}
    From Lemma 2.2 in \cite{B} and elliptic estimates, we have
    \begin{equation}
        \int_{\mathbb R^2 \setminus \bigcup_{i=1}^n B_i(t,r)} |\nabla \Phi_{0,t}|^2dx = \int_{\mathbb R^2 \setminus \bigcup_{i=1}^n B_i(t,r)} |\nabla \Phi_{r,t}|^2dx + o_r(1).
    \end{equation}
    Then by the definition of $W$, we obtain
     \begin{align}
        \frac{d}{dt}|_{t=0} E_\varepsilon (v_\varepsilon) & =  \frac{d}{dt}|_{t=0} \lim_{r \rightarrow 0} \frac{1}{2} \int_{\mathbb R^2 \setminus \bigcup_{i=1}^n B_i(t,r)} |\nabla \Phi_{r,t}|^2dx + o_\varepsilon(1) \notag \\
                                                                            & =  \frac{d}{dt}|_{t=0} W(\mathbf b(t)) + o_\varepsilon(1),
    \end{align}
    hence the desired result \eqref{h3}.

    The verification of \eqref{h4} is again analogous to 
the argument found in \cite{S} with appropriate adjustments 
as were just done in verifying \eqref{h3}.
\end{proof}

\begin{proof} [Proof of Theorem \ref{main}]
     Suppose, by contradiction, that there exists a sequence 
of stable critical points $\{u_\varepsilon\}$ such that 
$E_\varepsilon(u_\varepsilon) \leq C |\log \varepsilon|$, 
and, up to extraction, $n$ limiting vortices $b_1, b_2, ...,b_n$ 
with say, $K(b_1) >0$.

    Let $V = (V^1,V^2)$ be an arbitrary vector in 
$\mathbb R^2$. Then since we are assuming 
$n_\varepsilon^- =0$, in view of Proposition 1 and Theorem 
\ref{S}, we must have $n^- =0$, i.e.
    \begin{equation}\label{V}
        \frac{d^2}{dt^2}|_{t=0} W(b_1+tV, b_2, ..., b_n) = \sum_{i,j=1,2} \frac{\partial^2 W_1}{\partial x_i x_j} (b_1) V^iV^j \geq 0,
    \end{equation}
    where
    \[
        W_1(x) = \pi d_1^2 f(x) - \pi \sum_{j=1}^{n} d_1d_j \log|x-b_j|.
    \]
    Since the second term of $W_1$ is harmonic, we have
    \[
        \bigtriangleup W_1(b_1) = \pi d_1^2 \bigtriangleup f(b_1).
    \]
    Noting that the Gauss curvature at $b_1$ is given by
    \[
        0 < K(b_1) = -\frac{\bigtriangleup f}{e^{2f}} (b_1),
    \]
    we deduce that $D^2 f (b_1)$ has at least one negative 
eigenvalue, which contradicts \eqref{V}. Hence, if 
$u_\varepsilon$ are stable, the number of limiting vortices 
is 0, i.e. for $\varepsilon$ small enough, 
$|u_\varepsilon| \geq \frac{1}{2}$ in $\cal M$. However, 
as was mentioned in Remark 2.1, this implies that 
$u_\varepsilon$ is a constant.
\end{proof}

Now, let $\cal M$ be the surface obtained by rotating a regular curve

\[
        \gamma (s) = (\alpha(s),0,\beta(s)), \quad 0\leq s \leq l, \quad \alpha(s) > 0 \mbox{ for } s \neq 0, l.
\]
about the $z$-axis, where $s$ is the arc length, i.e. $| \gamma '| = 1$. 
Furthermore, make the assumptions: 
\begin{equation}\label{K}
    \alpha (0)  = \alpha(l) =\beta'(0) = \beta'(l) =  0, \mbox{ and either } 
    \beta''(0) \neq 0 \mbox{ or } \beta''(l)\not=0.
\end{equation}
We will henceforth assume $\beta''(0)\not=0$, 
the other case being similar. Denoting by $\theta$ the 
rotation angle, we then have a parametrization of $\cal M$
\begin{equation}\label{para}
        \mathbf x (s , \theta)=(\alpha(s)\cos ( \theta ), \alpha(s)\sin ( \theta ), \beta(s)), \quad 0 \leq s \leq l, \quad 0\leq \theta \leq 2\pi,
\end{equation}
and the induced metric
\[
         g_{\cal M} = ds^2 + \alpha^2(s) d\theta^2.
\]
Note in particular that, for $\cal M = \cal S$$^2$, we have $\alpha(\phi) = \sin(\phi)$ and
\[
         g_{\mbox{$\cal S$$^2$}} = d\phi^2 + \sin^2(\phi) d\theta^2.
\]

Consider a map $F:\cal S$$^2 \rightarrow \cal M$ such that parameter 
values $(\phi, \theta)$ corresponding to a point $p \in {\cal S}^2$ are 
mapped to parameter values $(S(\phi), \theta)$ in \eqref{para} 
corresponding to the point $q = \mathbf x (S(\phi), \theta)$, 
where $S(0) = 0$ and $S(\pi) = l$. Then $F$ is conformal if $S$ solves
\begin{equation}\label{cf}
    S'(\phi)\sin(\phi) = \alpha(s(\phi))
\end{equation}
Finally, we reparametrize $\cal M$ by defining $\mathbf y: \mathbb R^2 \bigcup \{ \infty\} \rightarrow \cal M$ through
\begin{align}\label{gy}
    & \mathbf y(x_1,x_2) = (\alpha(S(\phi))\cos(\theta), \alpha(S(\phi))\sin(\theta), \beta(S(\phi))) \mbox{ for } (x_1, x_2) \in \mathbb R^2 \bigcup \{\infty \},\notag \\
    &\mbox{where} \notag \\
    & \phi = \cos^{-1}(\frac{r^2 -1}{1+r^2}) \in [0,\pi], \notag \\
    & \theta = \tan^{-1}(\frac{x_2}{x_1}), \notag\\
    & r^2 = x_1^2 + x_2^2.
\end{align}
In other words, $\mathbf y = F \circ P^{-1}$, where  $P$ is the 
stereographic projection from $\cal S$$^2$ to the $x_1x_2$ plane. 
Using \eqref{cf} and \eqref{gy} the induced metric is given by
\[
    \tilde{g}_{\cal M} = \alpha^2(S(\phi))\frac{1}{r^2}(dx_1^2 + dx_2^2) \equiv e^{2f}(dx_1^2 + dx_2^2),
\]
i.e.
\begin{equation} \label{cff}
    f = \ln (\alpha(S(\phi)) \frac{1}{r}).
\end{equation}
Let $A = (\alpha'(S(\phi))+1) \frac{1}{r^2} \geq 0$, $B = -\frac{\alpha^2(S(\phi))}{r^2}K_{\cal M}$. With a direct calculation we obtain
\begin{equation}
    D^2f = \left(\begin{array}{ll}
                -A + \frac{x_1^2}{r^2} (2A+B) & \frac{x_1x_2}{r^2} (2A+B)\\
                 \frac{x_1x_2}{r^2} (2A+B)     &  -A + \frac{x_2^2}{r^2} (2A+B)
                \end{array} \right).
\end{equation}
Hence
\begin{equation} \label{231}
    Tr(D^2f) = B \mbox{ and } det(D^2f) = -A^2 - AB.
\end{equation}

Suppose that there exists a sequence of stable critical points 
$\{u_\epsilon\}$ such that $E_\epsilon(u_\epsilon) \leq C |\log \epsilon|$, 
and, up to extraction, $n$ limiting vortices $b_1, b_2, ...,b_n$. 
From Theorem \ref{main}, necessarily $K_{\cal M} (b_i) \leq 0$ for all 
$i$. In particular, none of vortices are at infinity since $K_{\cal M}$ 
is positive at the north pole of a surface of  revolution with 
$\beta''(0) \neq 0$. Thus we can use \eqref{cff} as the parametrization on $\cal M$.

However, when $K_{\cal M} \leq 0$, we have 
$Tr(D^2f) = B = -\frac{\alpha^2(S(\phi))}{r^2}K_{\cal M} \geq 0$, 
and $det(D^2f) \leq 0$. If there exists a $b_i$ such that 
$det(D^2f)(b_i) < 0$, then $D^2f(b_i)$ must have a negative 
eigenvalue. On the other hand, assume that $det(D^2f)(b_i) = 0$ 
for all $i$. We observe from \eqref{231} that $det(D^2f) =0$ if 
and only if $A=0$ i.e. $\alpha' = -1$ which implies that $B=0$ 
for the principle curvature in $\hat{\theta}$ direction is 0. Hence 
in this case, $D^2 f(b_i) = 0$ for all $i$, and the second variation 
of W only involves second derivatives of the log term given by
\begin{equation}
    -\pi \frac{d^2}{dt^2}|_{t=0}\sum_{i \neq j} d_id_j \log|b_i-b_j - t(V_i-V_j)| = \pi \sum_{i \neq j} d_id_j
\end{equation}
if we choose $V_i = b_i$. Then $\sum_{i=1}^n d_i =0$ implies that
\[
    \sum_{i \neq j} d_id_j = - \frac{1}{2} \sum_{i=1}^n d_i^2 <0.
\]
We have proved :

\begin{thm} \label{thm3}
    Let  $\cal M$ be a surface of revolution satisfying 
\eqref{K}, and $\{u_\varepsilon \}$ be a family of nonconstant 
solutions to \eqref{GL} such that 
$E_\varepsilon (u_\varepsilon) \leq C |\log \varepsilon|$.
Then for $\varepsilon$ small enough, $u_\varepsilon$ is unstable.
\end{thm}

\begin{rk}
    From Serfaty's result (\cite{S}) on bounded simply connected 
domains in $\mathbb R^2$ and the example of surfaces of revolution, 
we conjecture that Theorem \ref{thm3} holds for any simply connected 
compact surface, regardless of curvature conditions.
\end{rk}

\section{Vortex Annihilation}

In this section we look for conditions on a manifold that will imply
the ultimate annihilation of vortices under the Ginzburg-Landau heat
flow. To this end, let $(\cal M$$,g)$ be a smooth $2$-manifold with
boundary and consider the initial-boundary value problem
\begin{equation} \label{DGL}
                \left\{\begin{array}{ll}
                 u_t - \bigtriangleup _{\cal M} u = (1-|u|^2)u & \mbox{in $\cal M \times \mathbb R_+$} \\
                 u = e & \mbox{on $\partial \cal M \times \mathbb R_+$} \\
                 u = u_0 & \mbox{on $\cal M \times $$\{ t = 0 \} $}
                \end{array} \right.
\end{equation}
where for convenience we will associate $\mathbb C$ with 
$\mathbb R^2$ and consider 
$u:{\cal M} \times \mathbb R_+ \rightarrow \mathbb R^2$. 
Here $e$ is a constant unit vector and the initial data $u_0$ is 
allowed to have any number of vortices as long as their total 
degree satisfies $\sum d_i =0$.

Existence and regularity are standard for this problem:
\begin{prop}
If $u_0 \in W^{k,p}(\cal M)$ with $k > 2 + \frac{2}{p}$ for some $1 \leq p < \infty$ and $u_0=e$ on $\partial {\cal M}$, then \eqref{DGL} has a solution that exists for all time that is uniformly bounded. Furthermore, for each $T>0$,
\begin{align}
        &||u||_{C([0,T),W^{k,p}(\cal M))} \leq C(||u||_{L^{\infty}}) \label{P21} \\
        &||u||_{C^{1}([0,T),W^{k-2, p}(\cal M))} \leq C(||u||_{L^{\infty}}) \label{P22}
\end{align}
\end{prop}

\begin{proof}
This follows from Proposition 4.2 and 4.3 in \cite{T}.
\end{proof}

From the gradient flow structure of \eqref{DGL} we also easily
establish:
\begin{prop}
For each $T>0$,
\begin{align}
    \int_0^T\int_{\cal M} & |u_t|^2dv_g + \int_{\cal M}[\frac{||\nabla_g u||^2_g}{2} + V(u)](\cdot ,T)dv_g \notag \\
                                      &= \int_{\cal M}[\frac{||\nabla_g u_0||^2_g}{2} + V(u_0)]dv_g
\end{align}
where $V(u) = \frac{1}{4}(1-|u|^2)^2$.
\end{prop}

\begin{proof}
Taking inner product of \eqref{DGL}-1 with $u_t$ and integrating 
over $\cal M$ for a fixed $t$, we have
\begin{align}
        \int_{\cal M}|u_t|^2dv_g &= \int_{\cal M}u_t \cdot [\bigtriangleup _{\cal M}u + (1-|u|^2)u]dv_g \notag \\
                                               &= -\int_{\cal M}\langle \nabla_g u_t, \nabla_g u \rangle_g + V(u)_t dv_g \notag \\
                                               &= -\int_{\cal M}[\frac{||\nabla_g u||^2_g}{2}+V(u)]_tdv_g \notag
\end{align}
Integrating from $0$ to $T$, we get the desired equality.
\end{proof}

\begin{lem}
Suppose $p>2$ and $u_0$ satisfies the assumption of Proposition 3.1. 
Then for any sequence $\{t_n\}$ with $t_n \rightarrow \infty$ as 
$n \rightarrow \infty$, there exists a subsequence $\{t_{n_j}\}$ 
and a function $\bar{u}$ such that
\[
        u(\mathbf x,t_{n_j}) \rightarrow \bar{u}(\mathbf x) \quad in \quad C^2(\cal \bar{M}),
\]
and
\begin{equation} \label{lem1}
                \left\{\begin{array}{ll}
                 - \bigtriangleup _{\cal M} \bar{u} = (1-|\bar{u}|^2)\bar{u} & \mbox{in $\cal M$} \\
                \bar{u} = e & \mbox{on $\partial \cal M$}
                \end{array} \right.
\end{equation}
\end{lem}

\begin{proof}
From \eqref{P21} of Proposition 3.1, the sequence $\{u(\cdot ,t_n)\}$ 
is uniformly bounded in $W^{k,p}$. So by the Sobolev embedding 
theorem, there is a subsequence $\{ t_{n_j} \}$ and a $C^2$ 
function $\bar{u}(\mathbf x)$ such that
\[
        u(\mathbf x,t_{n_j}) \rightarrow \bar{u}(\mathbf x) \quad in \quad C^2(\cal \bar{M}).
\]
\par
To prove $\bar{u}(\mathbf x)$ is a solution of \eqref{lem1}, 
first we show that $\lim_{t \rightarrow \infty}||u_t||_{L^\infty(\cal M)}= 0$. 
Assume by way of contradiction that there is a sequence 
$\{(\mathbf {x_n},t_n)\}$ with $t_n \rightarrow \infty$ such that 
$|u_t(\mathbf{x_n},t_n)|>\epsilon>0$. Since \eqref{P22} of 
Proposition 3.1 implies that $u_t$ is uniformly continuous, 
there exists a $\delta > 0$ so that for all $n$, we have
\[
|u_t(\mathbf x,t)|>\frac{\epsilon}{2} \quad for \quad (\mathbf x,t) \in B_{\delta}(\mathbf {x_n}) \times (t_n- \delta , t_n + \delta),
\]
where $B_{\delta}(\mathbf {x_n})$ is the geodesic ball in 
$\cal M$ centered at $\mathbf {x_n}$ with radius $\delta$. 
But then $\int_0^{\infty}\int_{\cal M}|u_t|^2dv_g  = \infty$ 
which contradicts Proposition 3.2. Now, taking the limit as 
$j \rightarrow \infty$ in \eqref{DGL}-1 at time $t_{n_j}$, 
we get \eqref{lem1}.
\end{proof}

\begin{lem}
Suppose v solves \eqref{lem1} such that $v(\mathbf x) \neq 0$ 
for all $\mathbf x$ in $\cal M$. Then v is a constant.
\end{lem}

\begin{proof}
Write $e = e^{i \alpha_0}$ for some $\alpha_0 \in [0,2\pi)$. Then
from the assumption $v(\mathbf x) \neq 0$, we may write the function
$\tilde v = ve^{-i\alpha_0}$ in the form
\[
        \tilde v = \rho e^{i\alpha}
\]
for some smooth functions $\rho :\cal M \rightarrow \mathbb R_+$, 
and $0\leq \alpha < 2\pi$. Plugging this form into \eqref{lem1}, 
we have
\begin{equation} \label{lem2}
                 \left\{\begin{array}{ll}
                 \bigtriangleup _{\cal M} \rho - \rho ||\nabla_g \alpha||^2_g = \rho ^3 - \rho & \mbox{in $\cal M$} \\
                 \rho \bigtriangleup _{\cal M} \alpha  + 2\langle \nabla_g \rho ,\nabla_g \alpha \rangle_g = 0 & \mbox{in $\cal M$} \\
                 \rho = 1 & \mbox{on $\partial \cal M$} \\
                 \alpha = 0 & \mbox{on $\partial \cal M$} \\
                 \end{array} \right.
\end{equation}
Multiplying \eqref{lem2}-2 by $\alpha \rho$ and integrating, we get
\begin{align}
        0 & =  \int_{\cal M} \alpha [\rho^2 \bigtriangleup _{\cal M} \alpha  + 2\rho \langle \nabla \rho ,\nabla \alpha \rangle_g] dv_g \notag \\
           & = \int_{\cal M} \alpha \; \mathbf {div} (\rho^2 \nabla \alpha) dv_g \notag \\
           & = -\int_{\cal M} \rho^2|\nabla \alpha|^2_g dv_g \notag
\end{align}
Thus $\nabla \alpha = 0$ on $\cal M$, and so $\alpha\equiv 0$.

Then from \eqref{lem2}-1, since $\rho = |v| \geq 0$, we conclude 
by the maximum principle that $\rho \equiv 1$ on $\cal M$. 
This proves the lemma.
\end{proof}

In the last section, we deduced that the linear instability of
nonconstant critical points of the  Ginzburg-Landau energy for a
surface of revolution is independent of any curvature assumptions.
Now we will derive a result that is similar in spirit for the
parabolic problem \eqref{DGL} posed on a surface of revolution.
Consider a surface $\cal M$ with boundary defined parametrically as
in the last section:
\[
        \mathbf x ( s , \theta)=(\alpha(s) \cos (\theta), \alpha(s) \sin (\theta), \beta(s)), \quad 0 \leq s \leq l, \quad 0 \leq \theta \leq 2\pi
\]
with $\alpha(0) = \beta(0) = \beta'(0) = 0$, and $\alpha(s) > 0$ for $s \neq 0$. 
Note that $\partial {\cal M}=\{\mathbf x(l,\theta): 0\leq \theta \leq 2\pi\}$. 
We recall that the induced metric is
\[
    g = ds^2 + \alpha^2(s) d\theta^2.
\]

Now we present a crucial lemma which one can view as a 
kind of parabolic Pohozaev identity for heat flow on a manifold, 
cf. \cite{BCPS}, Lemma 4.1.

\begin{lem} \label{Pohozaev}
Let $H(s) = \int_0^s \alpha(\tilde{s})d\tilde{s}$, and let 
$\tilde{H}: \cal M \rightarrow \mathbb R$ be defined for any 
$p = \mathbf x (s,\theta) \in \cal M$ by the relation
$\tilde{H}(p) = H(s)$. Then for each $T>0$,
\begin{align}
        \int_0^T \int_{\cal M}& \tilde H|u_t|^2 + (\bigtriangleup_{\cal M} \tilde H)V(u)dv_gdt  + \int_{\cal M}\tilde H\left[\frac{||\nabla_g u||_g^2}{2}+V(u)\right](\cdot ,T)dv_g \notag \\
                                          &\leq \int_{\cal M}\tilde H\left[\frac{||\nabla_g u_0||_g^2}{2}+V(u_0)\right]dv_g.
\end{align}
\end{lem}

\begin{proof}
First, taking the inner product of \eqref{DGL}-1 with $\tilde Hu_t$ 
and integrating over $\cal M$ for a fixed $t$, we have
\begin{align} \label{lem51}
        \int_{\cal M}&\tilde H|u_t|^2dv_g = -\int_{\cal M} \langle \nabla_g u,\nabla_g (\tilde Hu_t) \rangle_gdv_g - \int_{\cal M} \tilde HV(u)_tdv_g +  \int_{\partial \cal M} \frac{\partial u}{\partial \mathbf n} \cdot \tilde H u_t \notag \\
                            &= - \int_{\cal M} \tilde H\left[\frac{||\nabla_g u||_g^2}{2}+V(u)\right]_tdv_g - \int_{\cal M} u_t \cdot \langle \nabla_g u,\nabla_g \tilde H\rangle_gdv_g,
\end{align}
where $\langle \nabla_g u, \nabla_g v\rangle_g \equiv \sum_{i=1}^2 \langle \nabla_g u_i, \nabla_g v_i\rangle_g$ for functions $u=(u_1,u_2)$, $v=(v_1,v_2)$ in $\mathbb R^2$. Here we have used the boundary condition $u = e$ on $\partial \cal M \times \mathbb R_+$, to chop the boundary integral.

Next, using \eqref{DGL}-1 and integrating by parts, we obtain
\begin{align}
         - \int_{\cal M} &u_t \cdot \langle \nabla_g u,\nabla_g \tilde H\rangle_gdv_g \notag \\
                                &= - \int_{\cal M} \bigtriangleup_{\cal M} u \cdot \langle \nabla_g u,\nabla_g \tilde H\rangle_gdv_g - \int_{\cal M} (1-|u|^2)u \cdot \langle \nabla_g u,\nabla_g \tilde H\rangle_gdv_g \notag \\
                                &= - \int_{\cal M} \bigtriangleup_{\cal M} u \cdot \langle \nabla_g u,\nabla_g \tilde H\rangle_gdv_g + \int_{\cal M} \langle \nabla_g V(u),\nabla_g \tilde H\rangle_gdv_g \notag \\
                                &= - \int_{\cal M} \bigtriangleup_{\cal M} u \cdot \langle \nabla_g u,\nabla_g \tilde H\rangle_gdv_g -  \int_{\cal M} (\bigtriangleup_{\cal M} \tilde H)V(u)dv_g
\end{align}
Integrating by parts twice in the first term on the right hand side yields

\begin{align} \label{lem53}
         - \int_{\cal M}& \bigtriangleup_{\cal M} u \cdot \langle \nabla_g u,\nabla_g \tilde H\rangle_gdv_g  \notag \\
                               &=\int_{\cal M} \frac{1}{2}\langle \nabla_g || \nabla u||^2_g,\nabla_g \tilde H\rangle_g + \alpha' || \nabla_g u||^2_gdv_g - \int_{\partial \cal M} \frac{\partial u}{\partial \mathbf n} \cdot \langle \nabla_g u,\nabla_g \tilde H\rangle_gdS  \notag \\
                               &=\frac{1}{2} \int_{\partial \cal M} ||\nabla_g u||^2_g \langle \nabla_g \tilde H,\mathbf n \rangle_g dS - \int_{\partial \cal M} \frac{\partial u}{\partial \mathbf n} \cdot \langle \nabla_g u,\nabla_g \tilde H\rangle_gdS
\end{align}
Note that $\nabla_g \tilde H = \alpha \mathbf n$ on $\partial \cal M$, \eqref{lem53} can be rewritten as

\begin{align} \label{lem54}
         \frac{1}{2} \int_{\partial \cal M}& ||\nabla_g u||^2_g \langle \nabla_g \tilde H,\mathbf n \rangle_g dS - \int_{\partial \cal M} \frac{\partial u}{\partial \mathbf n} \cdot \langle \nabla_g u,\nabla_g \tilde H\rangle_gdS \notag \\
                                                           &=\frac{1}{2} \int_{\partial \cal M} \alpha \left[|\frac{\partial u}{\partial \mathbf {\tau}}|^2_g - |\frac{\partial u}{\partial \mathbf n}|^2_g\right] dS \notag \\
                                                           &=-\frac{1}{2} \int_{\partial \cal M} \alpha |\frac{\partial u}{\partial \mathbf n}|^2_gdS
\end{align}
Combining \eqref{lem51}-\eqref{lem54}, since $\alpha \geq 0$, we have
\[
        \int_{\cal M}\tilde H|u_t|^2 + (\bigtriangleup_{\cal M} \tilde H)V(u)dv_g  + \int_{\cal M}\tilde H\left[\frac{||\nabla_g u||_g^2}{2}+V(u)\right]_tdv_g \leq 0
\]
Integrating from $0$ to $T$ gives the desired inequality.
\end{proof}

\begin{thm}\label{KoShin}
Assume $p>2$ and $u_0$ satisfies the assumption of Proposition 3.1. 
If $\alpha'(s) \geq c>0$ for $0 \leq s \leq l$, then $u(x,t) \rightarrow e$ 
uniformly as $t \rightarrow \infty$. In particular, $u$ has no vortices 
after some finite time $T$.
\end{thm}

\begin{proof}
Since $\bigtriangleup_{\cal M}\tilde H = 2\alpha'(s) \geq c > 0$, from Lemma 3.3, we have
\[
        \int_0^{\infty} \int_{\cal M}V(u)dv_gdt < \infty.
\]
Then arguing as in Lemma 3.1, we obtain
\begin{equation} \label{thm41}
        \lim_{t \rightarrow \infty} |u(x,t)|= 1 \mbox{ uniformly for } x \in {\cal M}.
\end{equation}
Now, \eqref{thm41} and Lemma 3.2 implies that $||u(\cdot ,t)-e||_{L^{\infty}} \rightarrow 0$
\end{proof}

\begin{rk}
From Proposition 3.1 there exists a $\delta > 0$ independent on $t$ 
such that $|u| > \frac{1}{2}$ in 
$\{\mathbf x(s,\theta) : \;  l - \delta < s \leq \l,  \; 0\leq \theta \leq 2\pi\}$. 
Thus the condition of $\alpha$ can be replaced by 
$\alpha'(s) \geq c>0$ for $0 \leq s \leq l - \delta$ in the theorem.
\end{rk}

\begin{rk}
The argument in the theorem does not involve the second 
derivative of $\alpha$. This indicates that the curvature does 
not affect the large-time behavior of the solution for this type 
of surface.
\end{rk}

\section*{Acknowledgment}
I would like to thank my adviser, Professor P. Sternberg, for his invaluable advice.

\end{document}